\documentclass[11pt, a4paper]{article}
\usepackage{amsmath,amssymb,amsthm,float,mathrsfs}

\usepackage[dvipdfmx]{graphicx,xcolor}
\usepackage{enumitem}

\newcommand{\E}{\mathbb{E}}
\renewcommand{\P}{\mathbb{P}}
\newcommand{\R}{\mathbb{R}}

\newcommand{\BP}{\mathscr{W}}
\newcommand{\cluster}{\mathscr{C}}
\renewcommand{\epsilon}{\varepsilon}
\newcommand{\constV}{C_\textup{V}}

\DeclareMathOperator\support{supp}

\newcommand{\pchain}{percolative chain }

\theoremstyle{definition}
\newtheorem{Def}{Definition}[section]
\newtheorem{rem}[Def]{Remark}

\theoremstyle{plain}
\newtheorem{thm}[Def]{Theorem}
\newtheorem{lem}[Def]{Lemma}

\newtheorem{prop}[Def]{Proposition}
\newtheorem{cor}[Def]{Corollary}

\makeatletter

\@addtoreset{equation}{section}
\makeatother

\title{Subcritical regimes in the Poisson Boolean percolation on Ahlfors regular spaces}
\author{Yutaka \textsc{takeuchi}\thanks{
    \begin{minipage}[t]{\textwidth}
        Faculty of Science and Technology, Keio University\;
        e-mail: \texttt{yutaka.takeuchi@keio.jp}
    \end{minipage}
}}
\date{}

\begin{document}
   \maketitle

   \begin{abstract}
    The Poisson Boolean percolation on a metric measure space is one of the percolation models. Intuitively, this model is obtained by collecting random balls whose centers form a Poisson point process.  
    In 2008, Gou\'{e}r\'{e} proved that for  $n \geq 2$, the Poisson Boolean percolation on $\R^n$ has the subcritical regime if and only if the radius distribution has finite $n$-th moment. In this paper, we extend Gou\'{e}r\'{e}'s result to Ahlfors regular metric measure spaces. 
   \end{abstract}

   \section{Introduction}
   \subsection{Motivation and historical backgrounds}
   In this paper, we study the relationship between geometric properties and the existence of subcritical regimes for Poisson Boolean percolation on metric measure spaces. Informally, the Poisson Boolean percolation is a union of metric balls centered at a Poisson point process with i.i.d. random radius. 
   The fundamental problem of percolation models is to study the existence of an infinite cluster for some parameter. We say that a percolation model has a supercritical regime if an infinite cluster exists for some parameter. On the other hand, we say that a percolation model has a subcritical regime if there is no infinite cluster for some parameter. If a percolation model has both the subcritical and supercritical regimes, we say that the phase transition occurs.  
   
   The Bernoulli bond percolation, introduced by Broadbent and Hammersley \cite{BroHam1957}, is the first percolation model. This model is obtained by removing edges of the lattice $\mathbb{Z}^d$ with same fixed probability and independently. This model is extended to other lattice models, Cayley graphs and trees. We refer the reader to \cite{Gri1999} and \cite{LyoPer2016} for a comprehensive treatment. Similarly, the Bernoulli site percolation is considered, which is obtained by removing sites.\\*\indent
   The Boolean percolation, also called the continuum percolation, is a continuum analogue of the Bernoulli site percolation. 
   The continuum percolation is first introduced by Gilbert \cite{Gil1961} as a random network on the plane.
   We recommend the book by Meester and Roy \cite{MeeRoy1996} for a comprehensive study. Hall \cite{Hall1985K} proved that the continuum percolation on $\R^n$ has the subcritical regime if the radius distribution has finite $(2n-1)$-moment, $n \geq 2$.
   Later, Gou\'{e}r\'{e} \cite{Gou2008} proved that the continuum percolation on $\R^n$ has the subcritical regime if and only if the radius distribution has finite $n$-moment, $n \geq 2$.

   For the study of the Poisson Boolean percolation on other metric spaces,  Tykesson studied the Poisson Boolean percolation with constant radii on the hyperbolic spaces. Tykesson \cite{Tyk2007} showed that there exist three regimes: In the first regime, there is no infinite occupied cluster and the number of infinite vacant cluster is infinite. In the second regime, there are infinitely many infinite clusters in both the occupied and the vacant region. In the third regime, There is a unique infinite occupied cluster and there is no infinite vacant cluster. 
   Tykesson \cite{Tyk2009} also examined the Poisson Boolean percolation with constant radii on a homogeneous space, that is, a Riemannian manifold for which there exists an isometry mapping between any two points. In \cite{Tyk2009}, Tykesson proved that if an unbounded cluster exists at intensity $\lambda_1$, then for any intensity $\lambda_2 > \lambda_1$, by coupling, the Boolean percolation at the intensity $\lambda_2$ contains the unbounded cluster at the intensity $\lambda_1$. 
   For more general metric space case, Coletti, Miranda and Mussini\cite{ColMirMus2016} showed that the subcritical and supercritical phases are invariant under mm-quasi-isometries. 

   The Bernoulli Boolean percolation, which is the discrete version of the Poisson Boolean percolation,  was studied by Coletti, Miranda and Grynberg in \cite{ColMirGry2020}. They found that Gou\'{e}r\'{e}'s result \cite{Gou2008} is applicable to the Bernoulli Boolean percolation on countable, locally finite, connected and doubling weighted graph under the doubling condition.

   In this paper, we extend Gou\'{e}r\'{e}'s results \cite{Gou2008}  to  general metric measure spaces which have some kind of global regularity, namely,  Ahlfors regularity. Our result is applicable to many examples. For example, it includes Riemannian manifolds, the unbounded Sierpinski gasket and ultrametric spaces and so on. See Section 4 below.
   
   \subsection{Settings and main results}
   Let $(S, d)$ be a locally compact, separable, and unbounded metric space and $\mu$ be a boundedly finite Borel measure (meaning every bounded Borel subset has finite $\mu$-measure) on $S$. Throughout this paper, we assume that $(S,d)$ is always locally compact, separable and unbounded.
   Let $\rho$ be a probability measure on $(0, \infty)$. 
   Let $\xi_\lambda$ be a Poisson point process on $S \times (0,\infty)$ with intensity measure $(\lambda \mu) \otimes \rho$, $\lambda > 0$, that is the random counting measure satisfying the following two properties: (a) for every Borel subset $A$ of $S\times (0,\infty)$, the distribution of the random variable $\xi_\lambda(A)$ is a Poisson distribution with parameter $\lambda\mu\otimes\rho(A)$, and, (b) if $A_1,\dots,A_n$ are disjoint Borel subsets, then $\xi_\lambda(A_1),\dots ,\xi_\lambda(A_n)$ are independent. We refer to the books \cite{LasPen2018}, \cite{DalV-Jon2003}, \cite{DalV-Jon2008} for the formal definition and more in-depth treatment of Poisson point processes.
   We denote the distribution of $\xi_\lambda$ by $\P_\lambda$. 
   The expectation with respect to the probability measure $\P_\lambda$ is denoted by $\E_\lambda$.
   
   Define the Poisson Boolean percolation $\BP \equiv \BP_\lambda(\mu, \rho)$ by 
   \begin{equation}
        \BP = \bigcup_{(x,R) \in \support \xi} B(x,R),
   \end{equation}
   where $B(x,R) = \{y \in S \mid d(x,y) < R\}$ is the metric open ball centered at $x$ with radius $R$. 
   A finite sequence of Poisson points $\{(z_i, R_i)\}_{i=0}^n \subset \support \xi$ is called a \pchain if $B(z_{i-1}, R_{i-1}) \cap B(z_i, R_i) \neq \emptyset$ for all $i = 1, \dots, n$. We say that two point $x$ and $y$ percolate each other and write $x \leftrightarrow y$ if there exists a \pchain $\{(z_i, R_i)\}_{i = 0}^n$ such that $x \in B(z_0, R_0)$ and $y \in B(z_n, R_n)$. Let $A, B_1, B_2$ be Borel subsets of $S$. We say that $B_1$ and $B_2$ percolate each other and write $B_1 \leftrightarrow B_2$ if there exist $x \in B_1$ and $y\in B_2$ such that $x \leftrightarrow y$. We say that $x$ and $y$ percolate each other in $A$ and write $x \leftrightarrow y \text{ in } A$ if there exists a \pchain $\{(z_i, R_i)\}_{i=0}^n$ such that $\{z_i\}_{i=0}^n \subset A$ and $x \in B(z_0 ,R_0)$ and $y \in B(x_n, R_n)$. We define $B_1 \leftrightarrow B_2$ in $A$ in a similar manner. Define a percolation cluster $\cluster(x)$ containing $x$ by $\cluster(x) = \{y \in S \mid x \leftrightarrow y\}$. Set $M(x) = \sup\{d(x,y) \mid y \in \cluster(x)\}$, where we use the convention $\sup\emptyset = 0$.  

   Next, we introduce the notion of the Ahlfors regular which plays a crucial role in the paper. 

   \begin{Def}
         Let $s > 0$. A Borel measure $\mu$ on $S$ is called $s$-Ahlfors regular if there exists $\constV \geq 1$ such that for all $x \in S$ and $r > 0$, we have 
        \begin{equation}\label{ine:AhlforsRegular}
            \constV^{-1} r^s \leq \mu(B(x,r)) \leq \constV r^s.
        \end{equation}
   \end{Def}

   We summarize a few properties of Ahlfors regular spaces. See \cite{MacTys2010} for example.
   \begin{lem}
        Let $(S, d, \mu)$ be an $s$-Ahlfors regular space. Let $\sigma > \constV^\frac{2}{\constV}$. Then  $(S, d)$ is $\sigma$-uniformly perfect, that is, $B(x, \sigma r) \setminus B(x, r)$ is nonempty for all $x \in S$ and $r > 0$. 
   \end{lem}

   \begin{lem}
        Let $(S, d, \mu)$ be an $s$-Ahlfors regular space. Then $(S,d)$ is (metric) doubling, that is, there exists $N$ such that any ball $B(x, r)$ can be covered by at most $N$ balls of radius $r/2$.
   \end{lem}

   Roughly speaking, the doubling property enables us to control the number of balls contained in a good covering. The uniformly perfectness tells us there is no ``large'' hole in the space. These two properties play important role in the proof of the existence of the subcritical regimes.

   We now state the main result in this paper.
   \begin{thm}\label{thm:Subcritical}
        Let $s > 0$. Let $(S,d,\mu)$ be an $s$-Ahlfors regular space. Let $\rho$ be a probability measure on $(0,\infty)$. 
        \begin{enumerate}[label=(\arabic*)]
            \item The followings are equivalent:  
            \begin{enumerate}[label=(\alph*)]
                \item[(1-a)] There exists $\lambda_0 >0$ such that for all  $\lambda \in (0,\lambda_0)$, we have
                \begin{equation*}\label{eqn:NoPercolation}
                    \lim_{r \to \infty}\sup_{x \in S} \P_\lambda(M(x) > r) = 0, 
                \end{equation*}
                \item[(1-b)] $\int_0^\infty R^s \rho(dR) < \infty$.
            \end{enumerate}
            \item Let $\beta >0$. Then the followings are equivalent:
            \begin{enumerate}[label=(\alph*)]
                \item[(2-a)] There exists $\lambda_0 >0$ such that for all  $\lambda \in (0,\lambda_0)$, we have
                \begin{equation*}
                    \sup_{x \in S} \E_\lambda[M(x)^\beta]< \infty, 
                \end{equation*}
                \item[(2-b)] $\int_0^\infty R^{s+\beta} \rho(dR) < \infty$.
            \end{enumerate}
        \end{enumerate}
   \end{thm}

   For the (1-b) $\Rightarrow$ (1-a) and (2-b) $\Rightarrow$ (2-a) parts, we will prove the following weaker version. As we will see later,  the lower bound in $(\ref{ine:AhlforsRegular})$ is used only for deducing the doubling property and the uniformly perfectness in the proof of the existence of the subcritical regimes. 
   \begin{thm}\label{thm:Subcritical2}
    Let $\sigma > 1$. Let $(S,d)$ be a doubling, $\sigma$-uniformly perfect metric space. Let $\mu$ be an infinite and boundedly finite measure on $S$. Let $\rho$ be a probability measure on $(0,\infty)$. Suppose that there exists $\constV \geq 1$ and $s > 0$ such that $\mu(B(x,r)) \leq \constV r^s$ holds for all $x\in S$, $r >0$. 
    \begin{enumerate}[label=(\arabic*)]
        \item  If $\int_0^\infty R^s\rho(dR) < \infty$, then there exists $\lambda_0 >0$ such that for all  $\lambda \in (0,\lambda_0)$, we have
        \begin{equation*}
            \lim_{r \to \infty}\sup_{x \in S} \P_\lambda(M(x) > r) = 0. 
        \end{equation*}
        \item If there exists $\beta >0$ such that $\int_0^\infty R^{s+\beta} \rho(dR)< \infty$, then There exists $\lambda_0 >0$ such that for all  $\lambda \in (0,\lambda_0)$, we have
        \begin{equation*}
            \sup_{x \in S} \E_\lambda[M(x)^\beta]< \infty.
        \end{equation*}
    \end{enumerate}
   \end{thm}
   \begin{rem}
      We remark that the above theorem does not guarantee the existence of the phase transition and the supercritical regime. Indeed, the one dimensional Euclidean space always lacks the supercritical regime whenever a radius distribution $\rho$ has the finite first moment. 
   \end{rem}

   Since the delta measure has finite $s$-moment, we can apply the above theorem to the constant radii case.  
   \begin{cor}\label{cor:SubcriticalForConstantRadii}
    Let $(S,d)$ be a doubling, uniformly perfect metric space. Let $\mu$ be an infinite and boundedly finite measure on $S$. Suppose that there exists $\constV \geq 1$ and $s > 0$ such that $\mu(B(x,r)) \leq \constV r^s$ holds for all $x\in S$, $r >0$. Let $\rho$ be the delta measure massed at some fixed radius $R_0 >0$. Then, there exists $\lambda_0 >0$ such that for all  $\lambda \in (0,\lambda_0)$, we have 
    $$\lim_{r \to \infty}\sup_{x \in S} \P_\lambda(M(x) > r) = 0. $$
\end{cor}
   The key ingredient of the proof of Theorem \ref{thm:Subcritical2} is estimating the probability that percolates a ball $B(x,r)$ to a ``surface" of a larger ball $B(x, R_1)$ within an even larger ball $B(x, R_2)$. However, an annulus $B(x,R_1) \setminus B(x, r)$ may be empty in general metric spaces. This is the main difference from Gou\'{e}r\'{e}'s work \cite{Gou2008}.
   In uniformly perfect spaces, when expanding a ball sufficiently, the annulus between the original and expanded balls is guaranteed to be non-empty, with a uniform dilation parameter.  This enables us to use Gou\'{e}r\'{e}'s technique after some modifications. 
   In addition, if $S$ has the doubling property, we can estimate the number of balls contained in a good covering. Hence, we can control the percolation probability of an annulus by using the union bound.

   For the proof of rest part of Theorem \ref{thm:Subcritical}, we will prove by contraposition. That is, we will prove the nonexistence of the subcritical regimes under the condition that $s$-moment or $s+\beta$-moment diverges. For the proof of the nonexistence of the subcritical regimes, we only use the volume control $(\ref{ine:AhlforsRegular})$. 

   This paper is organized as follows. In Section 2, we prove Theorem \ref{thm:Subcritical2}. In Section 3, we discuss about the nonexistence of the subcritical regime. Finally in Section 4, we provide some examples which are applicable to the main theorem.

   \section{Existence of subcritical regimes}
   In this section, we prove Theorem \ref{thm:Subcritical2}.
   Throughout this section, we assume that $(S,d)$ is doubling and $\sigma$-uniformly perfect, $\sigma > 1$. We further assume that there exists a constant $\constV \geq 1$ such that $\mu(B(x,r))\leq \constV r^s$ holds for all $x\in S$ and $r>0$. 
   We first introduce two events.
   Set 
   \begin{align*}
        &G(x, r) = \{B(x,r) \leftrightarrow B(x, 9\sigma^2r) \setminus B(x, 8\sigma r) \text{ in } B(x, 10\sigma^3r)\} \\
        &H(x, r) = \left\{
                \exists(y, R) \in \support \xi \text{ such that } y \in B(x, 10\sigma^3r)^c \text{ and } R > \frac{d(x,y)}{10\tau}
            \right\},
   \end{align*}
   where $\tau = \frac{\sigma}{10\sigma -9}$. 

   \begin{lem}\label{lem:ImplicationOfLongCluster}
    For all $x \in S$ and $r>0$, the following inclusion holds:
        \begin{equation*}
            \{M(x) > 9\sigma^2 r\} \subset G(x,r) \cup  H(x,r).
        \end{equation*}
   \end{lem}
   \begin{proof}
        We will show $G(x,r)^c \cap H(x,r)^c \subset \{M(x) \leq 9\sigma^2 r\}$. Suppose that $G(x,r)^c \cap H(x,r)^c$ occurs. Since $G(x,r)$ does not occur, any \pchain whose centers are contained in $B(x, 10\sigma^3r)$ does not connect $B(x, r)$ and $B(x, 9\sigma^2r) \setminus B(x, 8\sigma r)$. Next, observe that $\frac{10\tau -1}{10\tau} = \frac{9\sigma^2}{10\sigma^3} = \frac{9}{10}\sigma^{-1} \leq 1$. Since $H(x,r)$ does not occur, for all $(z, R) \in \support \xi$ with $z \in B(x, 10\sigma^3 r)^c$, we have 
        \begin{align*}
            d(x,z) - R \geq d(x,z) - \frac{d(x,z)}{10\tau} = \frac{10\tau -1}{10\tau}d(x,z) \geq \frac{9}{10}\sigma^{-1} \cdot 10\sigma^3 r = 9 \sigma^2 r.
        \end{align*}
        Hence for all $w \in B(z,R)$ we have $d(x,w) \geq d(x,z) - d(z, w) \geq d(x,z) - R \geq 9 \sigma_2 r$ and this implies $w \notin B(x, 9\sigma^2 r).$ Therefore, any percolative ball $B(z,R)$ centered at $B(x, 10\sigma^3r)^c$ cannot touch $B(x, 9\sigma^2r) \setminus B(x, 8\sigma r)$. Combining the above observations, we conclude that $M(x)$ must be less than or equal to $9 \sigma^2 r$.
    \end{proof}
    
   Since Lemma \ref{lem:Inclusion} holds, we need to estimate the probabilities of the events $G(x,r)$ and $H(x,r)$. 
   To estimate the probability of the event $G(x,r)$, we introduce a condition for a good covering.
   Let $x \in S$, $l,r > 0$ and $\epsilon \in (0,1)$.
   We say that a Borel subset $O\subset S$ satisfies the condition $(I_{x,l,r, \epsilon})$ if $O$ satisfies the followings:
   \begin{enumerate}[label=(\arabic*)]
    \item $\displaystyle O \subset B(x, \sigma r) \setminus B(x, r) \text{ and }  \bigcup_{y \in O}B(y,l) \supset B(x, \sigma r) \setminus B(x,  r)$,
    \item $O$ is $\epsilon l$-separated, that is, $d(y_1, y_2) \geq \epsilon l$ for each distinct $y_1, y_2 \in O$. 
   \end{enumerate}

   \begin{lem}\label{lem:KeyCovering}
        For all $x\in S$ and $r>0$, there are finite subsets $K(x,r)$ and $L(x,r)$ of $S$ that satisfy the conditions $(I_{x,r, 10\sigma^3 r, 1/5})$ and $(I_{x,r, 80\sigma^4 r, 1/5})$. Moreover, there exist constants $M_1, M_2 > 0$ such that for all $x \in S$ and $r>0$, we have $\# K(x,r) \leq M_1$ and $\# L(x,r) \leq M_2$.
   \end{lem}
   \begin{proof}
        First we consider $K(x,r)$. By the basic covering theorem (see \cite[Theorem 1.2]{Hei2001}), there exists $K(x, r) \subset B(x, 10\sigma^4 r) \setminus B(x, 10\sigma^3 r)$ such that 
        \begin{equation*}
            \bigcup_{y \in B(x, 10\sigma^4 r) \setminus B(x, 10\sigma^3 r)} B\biggl(y, \frac{r}{5}\biggr) \subset \bigcup_{y \in K(x,r)} B(y, r)
        \end{equation*}
        and $K(x,r)$ is $\frac{r}{5}$-separated. Since $(S,d)$ is doubling, by induction, there exist $\beta > 0$ and $C > 0$ such that for all $\epsilon \in (0,1)$, $z \in S$ and $r>0$, the maximal number of $\epsilon r$-separated point in $B(z,r)$ is less than or equal to $C\epsilon^{-\beta}$. (See \cite{MacTys2010}.) Set $M_1 = C\left(\frac{1}{50\sigma^4}\right)^{-\beta}$. Then we have that $\# K(x,r) \leq M_1$. We can show similarly for $L(x,r)$ with $M_2 = C\left(\frac{1}{80\sigma^5}\right)^{-\beta}$.
   \end{proof}
   

   Set 
   \begin{gather*}
    \widetilde{H}(x,r) = \Biggl\{\begin{split}
        &\text{there exists } (y,R) \in \support \xi \text{ such that } \\
        &y \in B(x, 100\sigma^6 r)\text{ and } R \geq r
    \end{split} 
    \Biggr\}.
    \end{gather*}
   \begin{lem}\label{lem:Inclusion}
        For all $x \in S$ and $r > 0$, we have
        \begin{equation*}
            G(x, 10\sigma^3 r) \setminus \widetilde{H}(x,r) \subset \bigcup_{y \in K(x,r)}G(y,r) \cap \bigcup_{z \in L(x,r)}G(z,r).
        \end{equation*}
   \end{lem}
   \begin{proof}
    Suppose that $G(x, 10\sigma^3 r)\setminus \widetilde{H}(x,r)$ occurs. Then there exists a \pchain $\{(w_i, r_i)\}_{i=0}^n$ that connects $B(x, 10\sigma^3 r)$ and $B(x, 90\sigma^5 r)\setminus B(x, 80\sigma^4 r)$. Then, 
    the \pchain $\{(w_i, r_i)\}_{i=0}^n$ touches one of a ball $B(y,r)$ centered at $K(x,r)$ with radius $r$ and one of a ball $B(z,r)$ centered at $L(x,r)$ with radius $r$. Since $\widetilde{H}(x,r)$ does not occur, we have $r_i \leq r$ for all $i = 0, \dots, n$. Therefore, $G(y,r)$ and $G(z,r)$ occur.
   \end{proof}
   
   \begin{prop}\label{prp:UniformRecursiveEstimate}
        Let $\lambda > 0$. For all $r >0$, we have
        \begin{equation*}
            \sup_{x \in S} \P_\lambda(G(x,10\sigma^3r)) \leq C_1 \biggl(\sup_{x\in S}\P_\lambda\bigl(G(x,r)\bigr)\biggr)^2 + \sup_{x \in S}\P_\lambda(\widetilde{H}(x,r)),
        \end{equation*}
        where $C_1 = M_1M_2$.
   \end{prop}
   \begin{proof}
        We first show that $\bigcup_{y \in K(x,r)} B(y,10\sigma^3r)$ and $\bigcup_{z \in L(x,r)} B(z,10\sigma^3r)$ are disjoint. Let $y \in K(x,r)$ and $z \in L(x,r)$ and take $u \in B(y,10\sigma^3r)$. Suppose that $d(u, z) \leq 10\sigma^3 r$. Then by the triangle inequality, we have $d(y,z) \leq 20\sigma^3r$. On the other hand, since $\sigma > 1$, we have $d(y,z) \geq d(x,z) - d(y,x) \geq 80\sigma^5 r - 10\sigma^4 r = (80\sigma - 10)\sigma^4 r > 70\sigma^4 r > 20\sigma^3 r$, a contradiction. Hence $d(u,z) > 10\sigma^3r$. Therefore, $u \notin B(z, 10\sigma^3 r)$ and this implies $B(y, 10\sigma^3 r) \cap B(z, 10\sigma^3 r) = \emptyset$. Since $y \in K(x,r)$ and $z\in L(x,r)$ are arbitrary, we conclude that $\bigcup_{y \in K(x,r)} B(y,10\sigma^3r)$ and $\bigcup_{z \in L(x,r)} B(z,10\sigma^3r) $ are disjoint. Now, by Lemma \ref{lem:Inclusion}, we have 
        \begin{equation}\label{ine:Inprp:URE:1}
            \P_\lambda (G(x,10\sigma^3r)) \leq \P_\lambda \left(\bigcup_{y \in K(x,r)}G(y,r) \cap \bigcup_{z \in L(x,r)}G(z,r)\right) + \P_\lambda(\widetilde{H}(x,r)).
        \end{equation}
        Since $\bigcup_{y \in K(x,r)}G(y,r)$ and  $\bigcup_{z \in L(x,r)}G(z,r)$ are disjoint, by Lemma \ref{lem:KeyCovering} and the union bound, we estimate 
        \begin{align}\label{ine:Inprp:URE:2}
            &\P_\lambda \left(\bigcup_{y \in K(x,r)}G(y,r) \cap \bigcup_{z \in L(x,r)}G(z,r)\right) \nonumber\\
            &= \P_\lambda\left(\bigcup_{y \in K(x,r)}G(y,r) \right) \P_\lambda \left( \bigcup_{z \in L(x,r)}G(z,r)\right) \nonumber \\
            &\leq M_1M_2 \left(\sup_{v \in S}\P_\lambda(G(v,r))\right)^2.
        \end{align}
        Combining $(\ref{ine:Inprp:URE:1})$ and $(\ref{ine:Inprp:URE:2})$ and taking the supremum, we obtain the desired result. 
   \end{proof}

   We remark that we only use the doubling property and the uniformly perfectness in the above Lemmas. In next three lemmas, we estimate the probabilities of the three events $G(x,r)$, $H(x,r)$, and $\widetilde{H}(x,r)$ by using the upper volume bound. 
   \begin{lem}
    \label{lem:EstimatonOfHtilde}
        For $\lambda, r >0$, we have 
        \begin{equation}
            \sup_{x \in S}\P_\lambda(\widetilde{H}(x,r)) \leq (100\sigma^6)^s \lambda \constV \int_r^\infty R^s \rho(dR). 
        \end{equation}
   \end{lem}
   \begin{proof}
       Set $V(x,r) = B(x, 100\sigma^6 r) \times [r, \infty)$. Then, we compute
       \begin{align*}
           \P_\lambda(\widetilde{H}(x,r)) 
           &= \P_\lambda(\xi(V(x,r)) \geq 1) \\
           &\leq \E_\lambda[\xi(V(x,r))] \\
           &= \lambda\mu \otimes \rho (V(x,r)) \\
           &\leq \lambda \int_r^\infty \mu(B(x, 100\sigma^6 R)) \rho(dR) \\
           &\leq \lambda \constV \int_r^\infty (100\sigma^6 R)^s \rho(dR).
       \end{align*}
       Taking the supremum, we get the desired result.
   \end{proof}

   \begin{lem}
    \label{lem:EstimationOfH}
        For $\lambda, r >0$, we have 
        \begin{equation}
            \sup_{x \in S}\P_\lambda(H(x,r)) \leq (10\tau)^s \lambda \constV \int_{\sigma^3r/\tau}^\infty R^s \rho(dR).
        \end{equation}
    \end{lem}
    \begin{proof}
        Set $U(x,r) = \{(y, R) \in S \times (0, \infty) \mid y \in B(x, 10\sigma^3 r)^c, R > \frac{d(x,y)}{10\tau}\}$. Observe that if $(y, R) \in U(x,r)$ then $y \in B(x, 10 \tau R)$ and $R > (\sigma^3 r )/\tau$. Then we estimate 
        \begin{align*}
            \P_\lambda(H(x,r)) 
            &= \P_\lambda(\xi(U(x,r)) \geq 1) \\
            &\leq \E_\lambda[\xi(U(x,r))]\\
            &\leq \lambda \int_{\sigma^3 r /\tau}^\infty \int_{B(x, 10\tau R)} 1\; \mu(dy) \rho(dR) \\
            &\leq \lambda \constV \int_{\sigma^3 r /\tau}^\infty (10\tau R)^s\rho(dR).
        \end{align*}
        Taking the supremum, the proof is over.
    \end{proof}

\begin{lem}
    \label{lem:EstimationOfG}
    Let $\lambda > 0$. For $r>0$, we have
   \begin{equation}
        \sup_{x \in S} \P_\lambda(G(x,r)) \leq (10\sigma^3)^s \lambda \constV r^s.
   \end{equation}
\end{lem}
\begin{proof}
    Let $\Pi \colon S \times (0,\infty) \to S$ be the projection and Set $\eta = \xi_\lambda \circ \Pi^{-1}$. Then $\eta$ is a Poisson point process with intensity measure $\lambda \mu$. Note that if $G(x,r)$ occurs then $\eta(B(x, 10\sigma^3 r)) \geq 1$. Therefore we have 
    \begin{align*}
        \P_\lambda(G(x,r)) &\leq \P_\lambda(\eta(B(x, 10\sigma^3 r)) \geq 1) \\
        &\leq \E_\lambda[\eta(B(x, 10\sigma^3 r))] \\
        &= \lambda \mu(B(x, 10\sigma^3 r)) \\
        &\leq (10\sigma^3 r)^s \lambda \constV,
    \end{align*} 
    we get the desired result after taking the supremum.
\end{proof}

    \begin{prop}
        Let $\lambda > 0$. Let $C_1>0$ be a constant appeared in Proposition \ref{prp:UniformRecursiveEstimate}. For $r > 0$,
        We have 
            \begin{equation}\label{ine:KeyScalingInequality} 
                \sup_{x\in S} \P_\lambda(G(x, 10\sigma^3r)) \leq C_1 (\sup_{x \in S} \P_\lambda(G(x, r)))^2 + (100\sigma^6)^s\lambda\constV\int_r^\infty R^s \rho(dR)
            \end{equation}
        and 
        \begin{equation}\label{ine:KeyClusterDistributionEstimation}
            \sup_{x\in S}\P_\lambda(M(x) > 9\sigma^2 r) \leq \sup_{x\in S}\P_\lambda(G(x,r)) + (10\tau)^s \lambda \constV\int_{\sigma^3 r /\tau}^\infty R^s \rho(dR).
        \end{equation}
    \end{prop}
    \begin{proof}
        The inequality $(\ref{ine:KeyScalingInequality})$ follows from Proposition \ref{prp:UniformRecursiveEstimate} and Lemma \ref{lem:EstimatonOfHtilde}. The inequality $(\ref{ine:KeyClusterDistributionEstimation})$ follows from Lemma \ref{lem:ImplicationOfLongCluster} and Lemma \ref{lem:EstimationOfH}.
    \end{proof}

    The following lemma is proved similarly as in \cite[Lemma 3.7.]{Gou2008}. (We only need replacing $10$ by $10c$.)
    \begin{lem}\label{lem:TechnicalLem}
       Let $c > 1$. Let $f,g \colon [1, \infty) \to [0,\infty)$ be bounded measurable functions satisfying the following properties:
       \begin{enumerate}[label=(\alph*)]
        \item $f|_{[1, 10c]} \leq \frac{1}{2}$ and $g \leq \frac{1}{4}$,
        \item for all $r \geq 10c$, $f(r) \leq f(\frac{r}{10c})^2 + g(r)$.
       \end{enumerate}
       Then $\lim_{r \to \infty}g(r) = 0$ implies $\lim_{r \to \infty} f(r) = 0$. Moreover, if there exists $\theta > 0$ such that the integral $\int_1^\infty r^{\theta-1}g(r) dr$ is finite, then the integral $\int_1^\infty r^{\theta-1}f(r) dr$ is also finite.
    \end{lem}

    \textit{Proof of the Theorem \ref{thm:Subcritical2}.}
    Set $c = \sigma^3$,  $f(r) = C_1\sup_{x\in S}\P_\lambda(G(x,r))$ and $g(r) = C_1(100\sigma^6)^s\lambda\constV\int_r^\infty R^s \rho(dR)$. Then by $(\ref{ine:KeyScalingInequality})$, we have $f(r) \leq f(\frac{r}{10c})^2 + g(r)$ for all $r \geq 10c$. Set 
    \begin{equation*}
        \lambda_0 = \min\left\{\frac{1}{2}(10\sigma^3)^sC_1\constV ,\:\, \frac{1}{4}(100\sigma^6)^s C_1\constV \int_0^\infty R^s \rho(dR)\right\}.
    \end{equation*}
    Let $\lambda \in (0, \lambda_0)$. Then by Lemma \ref{lem:EstimationOfG} and Lemma \ref{lem:EstimationOfH}, we have $f|_{[1, 10c]} \leq \frac{1}{2}$ and $g \leq \frac{1}{4}$. Moreover, $\lim_{r \to \infty}g(r) = 0$. Therefore, we can apply Lemma \ref{lem:TechnicalLem} and we have $\lim_{r \to \infty}f(r) = 0$. This implies that $\lim_{r \to \infty}\sup_{x \in S} \P_\lambda(G(x,r)) = 0$. By $(\ref{ine:KeyClusterDistributionEstimation})$, we have $\lim_{r \to \infty} \sup_{x \in S}\P_\lambda(M(x) > 9 \sigma^2 r) = 0$.
    Now we suppose that there exists $\beta>0$ such that $\int_0^\infty R^{s+\beta}\rho(dR) <\infty$. Let $\lambda \in (0, \lambda_0)$. Then by Fubini's theorem, we have 
    \begin{align*}
        \beta\int_1^\infty r^{\beta-1}\int_r^\infty R^s\rho(dR)dr = \int_1^\infty R^s(R^\beta -1) \rho(dR) \leq \int_0^\infty R^{s+\beta} < \infty. 
    \end{align*}
    Hence by Lemma \ref{lem:TechnicalLem}, we have that $\int_1^\infty r^{\beta-1}f(r)dr < \infty$.
    Since 
    \begin{align*}
        \sup_{x \in S}\E_\lambda[M(x)^\beta]
        &= \sup_{x\in S}\beta\int_0^\infty r^{\beta-1}\P_\lambda(M(x) \geq r) dr\\
        &\leq \beta\int_0^\infty r^{\beta-1}\sup_{x\in S}\P_\lambda(M(x) \geq r)dr,
    \end{align*}
    by $(\ref{ine:KeyClusterDistributionEstimation})$,  
    we conclude that $\sup_{x \in S}\E_\lambda[M(x)^\beta]$ is finite.

    \section{Nonexistence of the subcritical behavior}
    In this section, we show that the nonexistence of the subcritical behaviors. 
    \begin{lem}\label{lem:CoveringOneLargeBall}
       Let $(S,d,\mu)$ be $s$-Ahlfors regular. Let $\rho$ be a probability measure on $(0,\infty)$. Then for all $\lambda > 0$, $o \in S$, and $r>0$, we have 
       \begin{align*}
            &\P_\lambda(\text{There exists } (x,R) \in \support \xi \text{ such that } B(o, r) \subset B(x,R)) \nonumber \\
            &\geq 1 - \exp\biggl(- \frac{\lambda}{2^s\constV} \int_{2r}^\infty R^s \rho(dR) \biggr).
       \end{align*}
    \end{lem}
    \begin{proof}
        Set $A_r = \{(x, R) \in S \times (0,\infty) \mid B(o, r) \subset B(x, R)\}$. We have
        \begin{align*}
            \mu \otimes \rho (A_r) &\geq \mu \otimes \rho (\{(x, R) \mid d(o, x) < R-r\}) \\
            &= \int_r^\infty \int_{B(o, R-r)} 1\; \mu(dx)\rho(dR)\\
            &\geq  \int_{2r}^\infty \int_{B(o, R-r)} 1\; \mu(dx)\rho(dR)\\
            &\geq \int_{2r}^\infty \int_{B(o, \frac{R}{2})} \mu(dx)\rho(dR)\\
            &= \int_{2r}^\infty\mu\biggl(B\biggl(o, \frac{R}{2}\biggr)\biggr)\rho(dR)\\
            &\geq \frac{1}{2^s\constV}\int_{2r}^\infty R^s\rho(dR).
        \end{align*} 
        Therefore, we estimate that
        \begin{align*}
            &\P_\lambda(\text{There exists } (x,R) \in \support \xi \text{ such that } B(o, r) \subset B(x,R)) \\
            &=  \P_\lambda(\xi(A_r) \geq 1) \\
            &= 1 - \exp(-\lambda\mu\otimes\rho(A_r)) \\
            &\geq 1 - \exp\biggl(- \frac{\lambda}{2^s\constV} \int_{2r}^\infty R^s \rho(dR) \biggr).
        \end{align*}
    \end{proof}
    
    \begin{prop}\label{pro:WholeCover}
        Let $\lambda > 0$. Let $(S,d, \mu)$ be an $s$-Ahlfors regular space. Let $\rho$ be a probability measure on $(0, \infty)$. Then $\P_\lambda(\BP = S) = 1$ if and only if $\int_0^\infty R^s \rho(dR) = \infty$.
    \end{prop}
    \begin{proof}
        Fix a base point $o \in S$. 
        Suppose that $\int_0^\infty R^s \rho(dR) = \infty$. 
        Then, by Lemma \ref{lem:CoveringOneLargeBall}, we have $\P_\lambda(\xi(A_n) \geq 1) = 1$ for all $n \in \mathbb{N}$, hence we have $\BP \supset B(o, n)$ for all $n \in \mathbb{N}$ for almost surely. This implies that $\P_\lambda(\mathscr{W} = S) = 1$. Conversely, suppose that $\P_\lambda(\mathscr{W} = S) = 1$. Since $\{\mathscr{W} = S\} \subset \{o \in \mathscr{W}\}$, we have $\P_\lambda(o \in \mathscr{W}) = 1$. On the other hand, we have 
        \begin{align*}
            1 &= \P_\lambda(o \in \mathscr{W}) = \P_\lambda ( \exists (x, R)\in \support\xi \text{ such that } o \in B(x,R)) \\ 
            &= \P_\lambda(\exists (x, R)\in \support\xi \text{ such that } x \in B(o, R)) \\
            &= 1 - \exp{\left(-\int_0^\infty \int_{B(o,R)} 1 \; \mu(dx)\rho(dR)\right)}\\
            &= 1 - \exp{\left(-\int_0^\infty \mu(B(o,R))\rho(dR)\right)}\\
            &\leq 1- \exp{\left(-\constV\int_0^\infty R^s \rho(dR)\right)}.   
        \end{align*}
        Hence, we conclude $ \int_0^\infty R^s \rho(dR) = \infty$.
    \end{proof}
    
    Next, we prove the nonexistence of subcritical behavior in terms of the expected size of the cluster.  
    
    \begin{prop}\label{pro:NonexistenceSubcriticalBehavior}
       Let $(S,d,\mu)$ be an $s$-Ahlfors regular space. Let $\rho$ be a probability measure on $(0,\infty)$. Suppose that there exists a positive number $\beta > 0$ such that $\int_{0}^\infty R^{s+\beta}\rho(dR) = \infty$ holds. Then for all $\lambda > 0$ and $o \in S$, we have $\E_\lambda[M(o)^\beta] = \infty$. 
    \end{prop}
    \begin{proof}
        If $\int_0^\infty r^s \rho(ds) = \infty$, we know that $\BP = S$ almost surely and thus $\E_\lambda[M(o)^\beta] = \infty$. Therefore we consider the case that $\int_0^\infty r^s \rho(ds) < \infty$. 
        We first note that by Fubini's theorem, for all $p,q > 0$, the following identity
        \begin{align}\label{eqn:Cavalieri}
            \int_0^\infty R^{p+q} \rho(dR) = p\int_0^\infty  r^{p-1}\Biggl(\int_r^\infty R^q \rho(dR)\Biggr)dr
        \end{align}
        holds.
        Set $C = \frac{\lambda}{2^s\constV}\int_0^\infty R^s \rho(dR)$. By Lemma \ref{lem:CoveringOneLargeBall}, we estimate
        \begin{align}\label{ine:Inpro:NonexistenceSubcriticalBehavior:EstCoverdProb}
            &\P_\lambda(\exists(x,R) \in \support \xi \text{ such that } B(o, r) \subset B(x,R)) \nonumber \\
            &\geq 1 - \exp\biggl(- \frac{\lambda}{2^s\constV} \int_{2r}^\infty R^s \rho(dR) \biggr)\nonumber \\
            &= C^{-1}\Biggl(1 - \exp\biggl(-\frac{\lambda}{2^s\constV}\int_{2r}^\infty R^s \rho(dR)\biggr)\Biggr)\frac{\lambda}{2^s\constV}\int_{0}^\infty R^s\rho(dR)\nonumber \\
            &\geq C^{-1}\Biggl(1 - \exp\biggl(-\frac{\lambda}{2^s\constV}\int_{0}^\infty R^s\rho(dR)\biggr)\Biggr)\frac{\lambda}{2^s\constV}\int_{2r}^\infty R^s\rho(dR)\nonumber\\
            &= C^{-1}(1 - \exp(-C))\frac{\lambda}{2^s\constV}\int_{2r}^\infty R^s\rho(dR),
        \end{align}
        where we used the inequality $(1-e^{-a})b \geq (1-e^{-b})a$ for $0 \leq a \leq b$ in the forth line. Combining $(\ref{eqn:Cavalieri})$ and $(\ref{ine:Inpro:NonexistenceSubcriticalBehavior:EstCoverdProb})$, we have 
        \begin{align*}
            &\E_\lambda[M(o)^\beta] 
            = \beta\int_0^\infty r^{\beta-1}\P_\lambda(M(o)\geq r)dr \\
            &\geq \beta\int_0^\infty r^{\beta-1}\P_\lambda(\exists(x,R) \in \support \xi \text{ such that } B(o, r) \subset B(x,R))dr \\
            &\geq \beta \int_0^\infty r^{\beta-1}\Biggl(C^{-1}(1-e^{-C}) \frac{\lambda}{2^s\constV}\int_{2r}^\infty R^s\rho(dR) \Biggr)dr \\
            &= \frac{\beta \lambda C^{-1}(1-e^{-C})}{2^s\constV}\int_0^\infty r^{\beta-1}\Biggl(\int_{2r}^{\infty}R^s\rho(dR)\Biggr)dr\\
            &= \frac{\beta \lambda C^{-1}(1-e^{-C})}{2^{s+\beta}\constV}\int_0^\infty r^{\beta-1}\Biggl(\int_{r}^{\infty}R^s\rho(dR)\Biggr)dr\\
            &= \frac{\lambda C^{-1}(1-e^{-C})}{2^{s+\beta}(s+\beta)\constV} \int_0^\infty R^{s+\beta}\rho(dR) = \infty,
        \end{align*} 
        which is the desired result.
    \end{proof}
    Combining Theorem $\ref{thm:Subcritical2}$, Proposition \ref{pro:WholeCover}, and Proposition \ref{pro:NonexistenceSubcriticalBehavior}, we prove Theorem \ref{thm:Subcritical}.

    \section{Applications}
    In this section, we show that our results apply to a wide range of metric measure spaces. Several of these examples, to our knowledge, have not been previously appeared.
    \subsection{Poisson Boolean percolation on Riemannian manifold with nonnegative Ricci curvature}
    Let $S$ be a complete, connected, and unbounded $n$-dimensional Riemannian manifold with nonnegative Ricci curvature. Since $S$ is connected, $S$ is uniformly perfect. (See \cite[Section 1.3.2]{MacTys2010}.) 
    Let $\mu$ be the volume measure of $S$. 
    It is known that the volume growth of a Riemannian manifold with nonnegative curvature is bounded above by that of the Euclidean space (see \cite[Theorem III.4.4.]{Cha2006} for example). Moreover, by the Bishop-Gromov inequality, there exists $C>0$ such that 
    \begin{align*}
        \mu(B(x,R)) \leq C\biggl(\frac{R}{r}\biggr)^n \mu(B(x,r))
    \end{align*}
    holds for all $x\in S$, and  $r,R >0$. This implies that the volume doubling property $\mu(B(x,2r)) \leq 2^n \mu(B(x,r))$ and it follows that $(S,d)$ is doubling. (See \cite[Lemma B.3.4.]{Gro2007}.) 
    Therefore, we can apply Theorem \ref{thm:Subcritical2}.

    \subsection{Ultrametric space case}
    Let $(S, d, \mu)$ be an ultrametric and $s$-Ahlfors regular space, that is the ultratriangle inequality $d(x,z) \leq \max\{d(x,y), d(y,z)\}$ holds.   Combination of Theorem \ref{thm:Subcritical} and Proposition \ref{pro:WholeCover} shows that the Poisson Boolean percolation on $S$ is trivial.
    \begin{prop}
       Let $s>0$. Let $(S, d)$ be an ultrametric space. Let $\mu$ be an $s$-Ahlfors regular measure on $S$ and $\rho$ be a probability measure on $(0, \infty)$. If $\int_0^\infty R^s \rho(dR) < \infty$, then $\sup_{x\in S}\P_\lambda(M(x) = \infty) = 0$ for all $\lambda \in (0, \infty)$. If $\int_0^\infty R^s \rho(dR) = \infty$, then $\P_\lambda(\mathscr{W} = S) = 1$ for $\lambda \in (0,\infty)$.
    \end{prop}
    \begin{proof}
        We first note that if $d$ is ultrametric, then $y \in B(x,r)$ implies $B(y,r) = B(x,r)$. Moreover, if two balls have nonempty intersection, then one ball is contained in the other one. Therefore, for all $x \in S$ and $r > 0$, we have $\{M(x) > r\} = \{\xi(B(x,r)\times (r, \infty) ) \geq 1\}$.  
        Let $\lambda \in (0,\infty)$. We estimate 
        \begin{align*}
            \sup_{x \in S}\P_\lambda(M(x) > r) 
            &= \sup_{x \in S}\P_\lambda(\xi(B(x,r) \times (r,\infty) \geq 1 ))\\
            &= \sup_{x\in S} \Biggl(1 - \exp\biggl( - \lambda \int_r^\infty \mu(B(x,R))\rho(dR) \biggr)\Biggr)\\
            &\leq 1 - \exp\biggl( - \lambda \constV \int_r^\infty R^s\rho(dR)\biggr).
        \end{align*}
        If $\int_0^\infty R^s \rho(dR) < \infty$, the  right hand side of the above inequality tends to $0$ as $r \to \infty$. Therefore, we have $\lim_{r \to \infty}\sup_{x \in S}\P_\lambda(M(x) > r) = 0$.
        If $\int_0^\infty R^s \rho(dR) = \infty$, by Proposition \ref{pro:WholeCover}, we have $\P_\lambda(M(x) = \infty) = 1$ for all $\lambda \in (0,\infty)$. 
    \end{proof}

    \subsection{Unbounded Sierpinski gasket}
    Set $a_1 = (0,0), a_2 = (1,0)$, and $a_3 = (1/2, \sqrt{3}/2)$. For $i = 1,2,3$, define a map $F_i\colon \R^2 \to \R^2$  as $F_i(x) = (x-a_i)/2 + a_i$, $x \in \R^2$.
    Let $\tilde{S} \subset \R^2$ be the Sierpinski gasket, That is, the unique compact set satisfying $\tilde{S} = \bigcup_{i=1}^3 F_i(\tilde{S})$.
    Then the unbounded Sierpinski gasket $S$ is defined as $S = \bigcup_{n\geq 1}2^n \tilde{S}$.  Let $d(x,y)$ be the infimum of the length of the continuum paths in $S$ connecting $x$ and $y$. Let $\mu$ be the $(\log 3/\log 2)$-dimensional Hausdorff measure. Then $(S, d, \mu)$ is $(\log3 /\log2)$-Ahlfors regular. (See \cite[Remark 3.3]{Bar1998B}.) Hence we can apply Theorem \ref{thm:Subcritical}.

    \subsection{Inhomogeneous continuum percolation on the Euclidean space}
    Consider the $n$-dimensional Euclidean space $\R^n$ and let $d$ be the Euclidean distance. Let $g:\R^n \to [0,\infty)$ be a locally integrable function satisfying $\constV^{-1}r^n \leq  \int_{B(x,r)} g(x)dx \leq \constV r^n$. Set $d\mu = gdx$. Then the Poisson Boolean percolation $\BP_\lambda(\mu, \rho)$ has a subcritical regime. 
    Since $(\R^n, d, \mu)$ is mm-quasi isometric to $(\R^n ,d, dx)$, Theorem 1 of Coletti-Miranda-Mussini \cite{ColMirMus2016} implies that the Poisson Boolean percolation $\BP_\lambda(\mu,\rho)$ has subcritical regime if $\rho$ is the delta measure massed at some fixed radius. Since $(\R^n,d,\mu)$ is $n$-Ahlfors regular, we know that we can improve the above statement: for a radius distribution $\rho$, the Poisson Boolean percolation $\BP_\lambda(\mu, \rho)$ has the subcritical regime if and only if the radius distribution $\rho$ has finite $n$-th moment.

    \subsection{Snowflake of an Ahlfors regular space}
    Let $\alpha \in (0,1)$. For a metric $d$ on $S$, we can define a new metric $d^\alpha$ by $d^\alpha(x,y) = d(x,y)^\alpha$. The metric space $(S, d^\alpha)$ is called the snowflake of the metric space $(S, d)$. (See \cite[Definition 1.2.8.]{MacTys2010})  If $(S, d, \mu)$ is $s$-Ahlfors regular, then $(S, d^\alpha, \mu)$ is $s/\alpha$-Ahlfors regular. Therefore, by Theorem \ref{thm:Subcritical}, if $\rho$ is a probability measure on $(0, \infty)$ which has the finite $s/\alpha$-moment, the Poisson Boolean percolation has a subcritical regime. Note that because $(S, d)$ and $(S, d^\alpha)$ are not quasi-isometric, this example is beyond the reach of Coletti, Miranda, and Mussini's work \cite{ColMirMus2016}. Therefore, this is a new example that emerged from this study.

    $ $

    \textbf{Acknowledgements}
    The author thank Professor Hideki Tanemura for reading an earlier version of the manuscript and for useful comments.

   \bibliographystyle{hplain}
   \bibliography{PBMreference}
\end{document}